\documentclass[11pt]{amsart}
\usepackage{amssymb}
\usepackage{graphicx}                            
\usepackage{hyperref}
\usepackage{setspace}                            

\newtheorem{lemma}{Lemma}[section]
\newtheorem{theorem}{Theorem}[section]

\newenvironment{adfenumerate}{
\begin{enumerate}
\setlength{\itemsep}{0.5mm}
\setlength{\parskip}{0mm}
\setlength{\parsep}{0mm}
}{
\end{enumerate}
}

\newcommand{\adfQED}{\hfill $\square$}           
\newcommand{\adfmod}[1]{~(\mathrm{mod}~#1)}
\newcommand{\adfhide}[1]{}
\newcommand{\ADFvfyParStart}[1]{}
\newcommand{\adfDgap}{\vskip 2mm}              
\newcommand{\adfLgap}{\vskip 0mm}              
\newcommand{\adfsplit}{\par}                   
\newcommand{\adfBfont}{\normalsize}            

\begin{document}
\title{Group divisible designs with block size five}
\author{A. D. Forbes}
\address{LSBU Business School,
London South Bank University,
103 Borough Road,
London SE1 0AA, UK.}
\email{anthony.d.forbes@gmail.com}
\date{\today}

\begin{abstract}
We report some group divisible designs with block size five, including types $6^{15}$ and $10^{15}$.
As a consequence we are able to extend significantly the known spectrum for 5-GDDs of type $g^u$.
\end{abstract}

\maketitle


\section{Introduction}\label{sec:Introduction}
For the purpose of this paper, a {\em group divisible design}, $K$-GDD, of type $g_1^{u_1} g_2^{u_2} \dots g_r^{u_r}$ is an ordered triple ($V,\mathcal{G},\mathcal{B}$) such that:
\begin{adfenumerate}
\item[(i)]{$V$
is a base set of cardinality $u_1 g_1 + u_2 g_2 + \dots + u_r g_r$;}
\item[(ii)]{$\mathcal{G}$
is a partition of $V$ into $u_i$ subsets of cardinality $g_i$, $i = 1, 2, \dots, r$, called \textit{groups};}
\item[(iii)]{$\mathcal{B}$
is a non-empty collection of subsets of $V$ with cardinalities $k \in K$, called \textit{blocks}; and}
\item[(iv)]{each pair of elements from distinct groups occurs in precisely one block but no pair of
elements from the same group occurs in any block.}
\end{adfenumerate}
We abbreviate $\{k\}$-GDD to $k$-GDD,
and a $k$-GDD of type $q^k$ is also called a {\em transversal design}, TD$(k,q)$.
A {\em pairwise balanced design}, $(v, K, 1)$-PBD, is a $K$-GDD of type $1^v$.

A {\em parallel class} in a group divisible design is a subset of the block set that partitions the base set.
A $k$-GDD is called {\em resolvable}, and is denoted by $k$-RGDD, if the entire set of blocks can be partitioned into parallel classes.
If there exist $k$ mutually orthogonal Latin squares (MOLS) of side $q$, then there exists
a $(k+2)$-GDD of type $q^{k+2}$ and
a $(k+1)$-RGDD of type $q^{k+1}$, \cite[Theorem III.3.18]{AbelColbournDinitz2007}.
Furthermore, as is well known, there exist $q - 1$ MOLS of side $q$ whenever $q$ is a prime power.

Because of their widespread use in design theory,
especially in the construction of infinite classes of combinatorial designs by means of
the technique known as Wilson's Fundamental Construction,
\cite{WilsonRM1972}, \cite[Theorem IV.2.5]{GreigMullen2007},
group divisible designs are useful and important structures.
The existence spectrum problem for group divisible designs with constant block sizes, $k$-GDDs, $k \ge 3$,
appears to be a long way from being completely solved.
Nevertheless, for $k \in \{3, 4, 5\}$ where all the groups have the same size, considerable progress has been made.

The necessary conditions for the existence of $k$-GDDs of type $g^u$, namely
\begin{equation}\label{eqn:k-GDD g^u necessary}
\begin{array}{rcl}
u &\ge& k,\\
g(u - 1) &\equiv& 0 \adfmod{k - 1}, \\
g^2 u (u - 1) &\equiv& 0 \adfmod{k(k - 1)},
\end{array}
\end{equation}
are known to be sufficient for $k = 3$, \cite{Hanani1975}, \cite[Theorem IV.4.1]{Ge2007},
and for $k = 4$ except for types $2^4$ and $6^4$, \cite{BrouwerSchrijverHanani1977}, \cite[Theorem IV.4.6]{Ge2007}.
For 5-GDDs of type $g^u$, a partial solution to the design spectrum problem has been achieved,
\cite{AbelAssaf2002},
\cite{AbelAssafBluskovGreigShalaby2010},
\cite{AbelGeGreigLing2009},
\cite{BennettChangGeGreig2004},
\cite{Ge2007},
\cite{GeLing2004c},
\cite{GeLing2005b},
\cite{Hanani1975},
\cite{Rees2000},
\cite{WeiGe2014},
\cite{YinLingColbournAbel1997}, and
for future reference, we quote the main result concerning 5-GDDs in the important paper of Wei and Ge, \cite{WeiGe2014},
which represents a considerable advance on \cite[Theorem IV.4.16]{Ge2007} in the Colbourn--Dinitz {\em Handbook}.
\begin{theorem}[Wei, Ge, 2014]\label{thm:5-GDD-Wei-Ge}
The necessary conditions (\ref{eqn:k-GDD g^u necessary}) for the existence of a $5$-$\mathrm{GDD}$ of type $g^u$ are sufficient
except for types $2^5$, $2^{11}$, $3^5$, $6^5$, and except possibly for:
\begin{enumerate}
\item[] $g = 3$ and $u \in \{45, 65\}$;
\item[] $g = 2$ and $u \in \{15, 35, 71, 75, 95, 111, 115, 195, 215\}$;
\item[] $g = 6$ and $u \in \{15, 35, 75, 95\}$;
\item[] $g \in \{14, 18, 22, 26\}$ and $u \in \{11, 15, 71, 111, 115\}$;
\item[] $g \in \{34, 46, 62\}$ and $u \in \{11, 15\}$;
\item[] $g \in \{38, 58\}$ and $u \in \{11, 15, 71, 111\}$;
\item[] $g = 2\alpha$, $\gcd(\alpha,30) = 1$, $33 \le \alpha \le 2443$, and $u = 15$;
\item[] $g = 10$ and $u \in \{5, 7, 15, 23, 27, 33, 35, 39, 47\}$;
\item[] $g = 30$ and $u = 15$;
\item[] $g = 50$ and $u \in \{15, 23, 27\}$;
\item[] $g = 90$ and $u = 23\}$;
\item[] $g = 10\alpha$, $\alpha \in \{7, 11, 13, 17, 35, 55, 77, 85, 91, 119, 143, 187, 221\}$ and $u = 23$.
\end{enumerate}
\end{theorem}
\begin{proof}
This is Theorem 2.25 of \cite{WeiGe2014}.
\end{proof}

The objective of this paper is to prove Theorem~\ref{thm:5-GDD-ADF}, below,
which improves Theorem~\ref{thm:5-GDD-Wei-Ge} by eliminating many possible exceptions.
\begin{theorem}\label{thm:5-GDD-ADF}
The necessary conditions (\ref{eqn:k-GDD g^u necessary}) for the existence of a $5$-$\mathrm{GDD}$ of type $g^u$ are sufficient
except for types $2^5$, $2^{11}$, $3^5$, $6^5$, and except possibly for:
\begin{enumerate}
\item[] $g = 3$ and $u = 65$;
\item[] $g = 2$ and $u \in \{15, 75, 95, 115\}$;
\item[] $g = 6$ and $u \in \{35, 95\}$;
\item[] $g \in \{14, 18, 22, 26, 38, 58\}$ and $u \in \{11, 15\}$;
\item[] $g \in \{74, 82, 86, 94\}$ and $u = 15$;
\item[] $g = 10$ and $u \in \{5, 7, 27, 39, 47\}$;
\item[] $g = 50$ and $u = 27$.
\end{enumerate}
\end{theorem}


\section{GDDs with block size 5 and type $g^u$}
\label{sec:GDDs-with-block-size-5-and-type-g-u}
We begin with some directly constructed group divisible designs.
%
\begin{theorem}\label{thm:5-GDD-g-u-direct}
There exist $5$-$\mathrm{GDD}$s of types $2^{35}$, $2^{71}$, $2^{111}$, $3^{45}$, $6^{15}$, $10^{15}$, $10^{23}$ and $10^{33}$.
\end{theorem}
\begin{proof}
For $2^{35}$, $2^{71}$, and $10^{23}$ see \cite[Lemma 4.1]{Forbes2020P}.
%
%

\adfhide{
$ 2^{111} $,
$ 6^{15} $,
$ 10^{15} $ and
$ 10^{33} $.
}

\adfDgap
\noindent{\boldmath $ 2^{111} $}~
With the point set $\{0, 1, \dots, 221\}$ partitioned into
 residue classes modulo $111$ for $\{0, 1, \dots, 221\}$,
 the design is generated from

\adfLgap {\adfBfont
$\{137, 73, 211, 182, 50\}$,
$\{138, 74, 212, 183, 51\}$,
$\{148, 201, 185, 107, 206\}$,\adfsplit
$\{149, 202, 186, 108, 207\}$,
$\{202, 148, 11, 152, 191\}$,
$\{203, 149, 12, 153, 192\}$,\adfsplit
$\{119, 166, 168, 153, 212\}$,
$\{120, 167, 169, 154, 213\}$,
$\{123, 106, 46, 71, 188\}$,\adfsplit
$\{124, 107, 47, 72, 189\}$,
$\{84, 132, 77, 65, 156\}$,
$\{0, 3, 12, 122, 136\}$,\adfsplit
$\{0, 8, 38, 126, 154\}$,
$\{0, 7, 83, 156, 219\}$,
$\{0, 10, 32, 101, 102\}$,\adfsplit
$\{0, 27, 55, 75, 182\}$,
$\{0, 33, 51, 57, 108\}$,
$\{0, 1, 107, 121, 204\}$,\adfsplit
$\{0, 79, 119, 151, 189\}$,
$\{1, 9, 31, 97, 123\}$,
$\{0, 6, 26, 62, 159\}$,\adfsplit
$\{0, 9, 71, 127, 195\}$

}
\adfLgap \noindent by the mapping:
$x \mapsto x + 2 j \adfmod{222}$,
$0 \le j < 111$.
\ADFvfyParStart{(222, ((22, 111, ((222, 2)))), ((2, 111)))} 

\adfDgap
\noindent{\boldmath $ 3^{45} $}~
With the point set $\{0, 1, \dots, 134\}$ partitioned into
 residue classes modulo $44$ for $\{0, 1, \dots, 131\}$, and
 $\{132, 133, 134\}$,
 the design is generated from

\adfLgap {\adfBfont
$\{121, 84, 8, 48, 108\}$,
$\{82, 9, 79, 86, 124\}$,
$\{133, 30, 56, 57, 35\}$,\adfsplit
$\{131, 80, 60, 9, 37\}$,
$\{95, 70, 122, 60, 91\}$,
$\{0, 2, 8, 30, 49\}$,\adfsplit
$\{0, 3, 18, 85, 115\}$,
$\{0, 12, 75, 77, 86\}$,
$\{0, 14, 53, 78, 93\}$,\adfsplit
$\{0, 16, 45, 50, 119\}$,
$\{0, 9, 43, 84, 95\}$,
$\{0, 7, 23, 83, 131\}$,\adfsplit
$\{1, 7, 19, 33, 97\}$,
$\{0, 33, 66, 99, 134\}$

}
\adfLgap \noindent by the mapping:
$x \mapsto x + 2 j \adfmod{132}$ for $x < 132$,
$x \mapsto (x +  j \adfmod{2}) + 132$ for $132 \le x < 134$,
$134 \mapsto 134$,
$0 \le j < 66$
 for the first 13 blocks,
$0 \le j < 33$
 for the last block.
\ADFvfyParStart{(135, ((13, 66, ((132, 2), (2, 1), (1, 1))), (1, 33, ((132, 2), (2, 1), (1, 1)))), ((3, 44), (3, 1)))} 

\adfDgap
\noindent{\boldmath $ 6^{15} $}~
With the point set $\{0, 1, \dots, 89\}$ partitioned into
 residue classes modulo $15$ for $\{0, 1, \dots, 89\}$,
 the design is generated from

\adfLgap {\adfBfont
$\{80, 41, 45, 18, 25\}$,
$\{0, 1, 41, 67, 88\}$,
$\{0, 21, 29, 63, 73\}$,\adfsplit
$\{0, 14, 39, 40, 71\}$,
$\{0, 5, 34, 81, 83\}$,
$\{0, 4, 13, 16, 84\}$,\adfsplit
$\{0, 8, 32, 52, 85\}$,
$\{0, 11, 17, 42, 79\}$,
$\{0, 18, 36, 54, 72\}$,\adfsplit
$\{1, 19, 37, 55, 73\}$

}
\adfLgap \noindent by the mapping:
$x \mapsto x + 2 j \adfmod{90}$,
$0 \le j < 45$
 for the first eight blocks,
$0 \le j < 9$
 for the last two blocks.
\ADFvfyParStart{(90, ((8, 45, ((90, 2))), (2, 9, ((90, 2)))), ((6, 15)))} 

\adfDgap
\noindent{\boldmath $ 10^{15} $}~
With the point set $\{0, 1, \dots, 149\}$ partitioned into
 residue classes modulo $15$ for $\{0, 1, \dots, 149\}$,
 the design is generated from

\adfLgap {\adfBfont
$\{101, 21, 43, 132, 59\}$,
$\{12, 85, 61, 88, 129\}$,
$\{29, 9, 85, 93, 147\}$,\adfsplit
$\{141, 39, 26, 48, 88\}$,
$\{7, 76, 86, 25, 110\}$,
$\{0, 1, 12, 108, 137\}$,\adfsplit
$\{0, 14, 32, 111, 145\}$,
$\{0, 17, 57, 63, 67\}$,
$\{0, 16, 84, 107, 143\}$,\adfsplit
$\{0, 2, 21, 102, 146\}$,
$\{0, 8, 86, 95, 112\}$,
$\{0, 7, 35, 36, 130\}$,\adfsplit
$\{0, 11, 37, 58, 109\}$,
$\{0, 3, 5, 70, 122\}$

}
\adfLgap \noindent by the mapping:
$x \mapsto x + 2 j \adfmod{150}$,
$0 \le j < 75$.
\ADFvfyParStart{(150, ((14, 75, ((150, 2)))), ((10, 15)))} 

\adfDgap
\noindent{\boldmath $ 10^{33} $}~
With the point set $\{0, 1, \dots, 329\}$ partitioned into
 residue classes modulo $33$ for $\{0, 1, \dots, 329\}$,
 the design is generated from

\adfLgap {\adfBfont
$\{102, 84, 56, 8, 268\}$,
$\{145, 251, 217, 214, 137\}$,
$\{57, 303, 73, 97, 184\}$,\adfsplit
$\{304, 149, 216, 134, 104\}$,
$\{203, 229, 88, 107, 278\}$,
$\{170, 150, 53, 139, 229\}$,\adfsplit
$\{300, 246, 79, 41, 278\}$,
$\{108, 129, 65, 133, 48\}$,
$\{0, 13, 120, 193, 222\}$,\adfsplit
$\{0, 7, 42, 65, 214\}$,
$\{0, 1, 148, 153, 162\}$,
$\{0, 27, 63, 110, 201\}$,\adfsplit
$\{0, 10, 62, 136, 197\}$,
$\{0, 2, 55, 105, 144\}$,
$\{0, 6, 57, 98, 202\}$,\adfsplit
$\{0, 12, 56, 151, 229\}$

}
\adfLgap \noindent by the mapping:
$x \mapsto x +  j \adfmod{330}$,
$0 \le j < 330$.
\ADFvfyParStart{(330, ((16, 330, ((330, 1)))), ((10, 33)))} 
%
%
\end{proof}

\vskip 1mm
For our proof of Theorem~\ref{thm:5-GDD-ADF}, we require some definitions and constructions.

A {\em double group divisible design}, $k$-DGDD, is an ordered quadruple ($V,\mathcal{G},\mathcal{H},\mathcal{B}$) such that:
\begin{enumerate}
\item[(i)]{$V$ is a base set of points;}
\item[(ii)]{$\mathcal{G}$ is a partition of $V$, the \textit{groups};}
\item[(iii)]{$\mathcal{H}$ is another partition of $V$, the \textit{holes};}
\item[(iv)]{$\mathcal{B}$ is a non-empty collection of subsets of $V$ of cardinality $k$, the \textit{blocks};}
\item[(v)]{for each block $B \in \mathcal{B}$, each group $G \in \mathcal{G}$ and each hole $H \in \mathcal{H}$,
we have $|B \cap G| \le 1$ and $|B \cap H| \le 1$;}
\item[(vi)]{%
each pair of elements of $V$ not in the same group and not in the same hole occurs in precisely one block.}
\end{enumerate}
A $k$-DGDD of type
$$(g_1, h_1^w)^{u_1} (g_2, h_2^w)^{u_2} \dots (g_r, h_r^w)^{u_r}, ~~ g_i = w h_i,~~ i = 1, 2, \dots, r,$$
is a double group divisible design where:
\begin{enumerate}
\item[(i)]{there are $u_i$ groups of size $g_i$, $i = 1$, 2, \dots, $r$;}
\item[(ii)]{there are $w$ holes;}
\item[(iii)]{for $i = 1$, 2, \dots, $r$, each group of size $g_i$ intersects each hole in $h_i$ points.}
\end{enumerate}
A {\em modified group divisible design}, $k$-MGDD, of type $g^u$ is a $k$-DGDD of type $(g, 1^g)^u$.
By interchanging groups and holes we see that a $k$-MGDD of type $g^u$ exists if and only if a $k$-MGDD of type $u^g$ exists.
See \cite{AbelAssaf2002} for an extensive treatment of $5$-MGDDs.
%
\begin{lemma}
\label{lem:5-GDD-g-to-gh}
Suppose there exists a $5$-$\mathrm{GDD}$ of type $g_1^{u_1} g_2^{u_2} \dots g_n^{u_n}$.
Then for positive integer $h \notin \{2, 3, 6, 10\}$,
there exists a $5$-$\mathrm{GDD}$ of type $(g_1 h)^{u_1} (g_2 h)^{u_2} \dots (g_n h)^{u_n}$.
\end{lemma}
\begin{proof}
Inflate each point of the 5-GDD by a factor of $h$ and replace the blocks with 5-GDDs of type $h^5$.
By Theorem~\ref{thm:5-GDD-Wei-Ge}, there exists a 5-GDD of type $h^5$ for $h \ge 1$, $h \not\in \{2, 3, 6, 10\}$.
\end{proof}
%
\begin{lemma}
\label{lem:K-GDD-5-MGDD-5-GDD-to-5-GDD}
Suppose there exists a $K$-$\mathrm{GDD}$ of type $g_1^{u_1} g_2^{u_2} \dots g_r^{u_r}$, and let $w$ be a positive integer.
Suppose also that for each $k \in K$, there exists a $5$-$\mathrm{MGDD}$ of type $w^k$, and
for $i = 1, 2, \dots, r$, there exists a $5$-$\mathrm{GDD}$ of type $g_i^w$.
Then there exists a $5$-$\mathrm{GDD}$ of type $(u_1 g_1 + u_2 g_2 + \dots + u_r g_r)^w$.
\end{lemma}
\begin{proof}
This is a combination of Constructions 2.19 and 2.20 in \cite{WeiGe2014}, and
it also appears (for block size 4) as Constructions 1.8 and 1.10 in \cite{GeReesZhu2002}.

Take the $K$-GDD and inflate each point by a factor of $w$.
Replace each inflated block by a 5-MGDD of type $w^k$, $k \in K$ to obtain a 5-DGDD of type
$$(w g_1, g_1^w)^{u_1} (w g_2, g_2^w)^{u_2} \dots (w g_r, g_r^w)^{u_r}.$$
Then overlay the holes of this 5-DGDD with 5-GDDs of types $g_i^w$, $i = 1, 2, \dots, r$.
%
\end{proof}

We can now prove our main result.\\

\noindent {\bf Proof of Theorem~\ref{thm:5-GDD-ADF}}.

\noindent For types $2^{35}$, $2^{71}$, $2^{111}$, $3^{45}$, $6^{15}$, $10^{15}$, $10^{23}$ and $10^{33}$, see Theorem~\ref{thm:5-GDD-g-u-direct}.

For types $2^{195}$ and $2^{215}$, take a 5-GDD of type $68^5 48^1$ or $68^5 88^1$,
\cite{Rees2000} (alternatively, see \cite[Theorem 2.1]{AbelGeGreigLing2009} or \cite[Theorem IV.4.17]{Ge2007}), and adjoin 2 extra points.
Overlay each group together with the new points with a 5-GDD of type $2^{25}$ or $2^{35}$ or $2^{45}$, as appropriate.

For type $6^{75}$, take a 5-GDD of type $90^5$ and overlay the groups with 5-GDDs of type $6^{15}$.

For type $g^{t}$, $g \in \{14, 18, 22, 26, 38, 58\}$, $t \in \{71, 111\}$,
use Lemma~\ref{lem:5-GDD-g-to-gh} with type $2^{71}$ or $2^{111}$ and $h = g/2$.

For type $g^{115}$, $g \in \{14, 18, 22, 26\}$,
construct a 5-GDD of type $(5g)^{23}$ using Lemma~\ref{lem:5-GDD-g-to-gh} with a 5-GDD of type $10^{23}$ and $h = g/2$;
then replace each group with a 5-GDD of type $g^5$.

For types $10^{35}$, $30^{15}$ and $50^{15}$, use Lemma~\ref{lem:5-GDD-g-to-gh}
with a 5-GDD of type $2^{35}$ or $6^{15}$ or $10^{15}$, as appropriate, and $h = 5$.

For type $(10\alpha)^{23}$, odd $\alpha \ge 5$, use Lemma~\ref{lem:5-GDD-g-to-gh} with a 5-GDD of type $10^{23}$ and $h = \alpha$.

Let
\begin{align*}
G &= \{34, 46, 62\} \cup \left\{\text{even }g \ge 66: \gcd\left(\dfrac{g}{2}, 30\right) = 1\right\}\\
  &\;\;\;\;\; \setminus \{74, 82, 86, 94, 98, 106, 118, 178\}.
\end{align*}
For $g \in G$, there exists a $(g + 1, \{5,7,9\}, 1)$-PBD, \cite[Table IV.3.23]{AbelBennettGreig2007}.
Take this PBD, remove a point and the blocks containing it to get a $\{5,7,9\}$-GDD of type $4^a 6^b 8^c$
for some non-negative integers $a$, $b$, $c$ satisfying $4a + 6b + 8c = g$.
Now use Lemma~\ref{lem:K-GDD-5-MGDD-5-GDD-to-5-GDD} with this $\{5,7,9\}$-GDD and
$w = 11$ or $15$ to obtain 5-GDDs of types $g^{11}$ and $g^{15}$ for every $g \in G$.
For the existence of 5-MGDDs of types $w^{5}$, $w^{7}$ and $w^9$, see \cite{AbelAssaf2002}.
For the existence of 5-GDDs of types $4^{w}$, $6^{w}$ and $8^{w}$, see Theorems~\ref{thm:5-GDD-Wei-Ge} and \ref{thm:5-GDD-g-u-direct}.

For type $98^{15}$, take a TD$(9, 11)$, fill in the groups with blocks of size 11 and remove a point together with the blocks containing it
to get a $\{9, 11\}$-GDD of type $8^{11} 10^1$.
Now use Lemma~\ref{lem:K-GDD-5-MGDD-5-GDD-to-5-GDD} with this $\{9, 11\}$-GDD and $w = 15$ to obtain a 5-GDD of type $98^{15}$.
For the existence of 5-MGDDs of types $15^{9}$ and $15^{11}$, see \cite{AbelAssaf2002}.
For the existence of 5-GDDs of types $8^{15}$ and $10^{15}$, see Theorems~\ref{thm:5-GDD-Wei-Ge} and \ref{thm:5-GDD-g-u-direct}.

For types $106^{15}$, $118^{15}$ and $178^{15}$, we refer the reader to Lemma 3.16 of \cite{GeLing2005b},
which proves that there exists a 5-GDD of type $h^{11}$ for $h \equiv 2 \adfmod{4}$, $h \ge 66$.
By \cite[Theorem 1.3]{GeLing2005b}, there exists a 4-frame of type $6^{15}$,
i.e.\ a $4$-GDD ($V,\mathcal{G},\mathcal{B}$) of type $6^{15}$ in which the block set can be partitioned into
into $30$ partial parallel classes of size $21$ each of which partitions $V \setminus G$ for some $G \in \mathcal{G}$.
Also we have the 5-GDD of type $6^{15}$ from Theorem~\ref{thm:5-GDD-g-u-direct} as well as
5-GDDs of type $h^{15}$ for $h \equiv 0 \adfmod{4}$ from Theorem~\ref{thm:5-GDD-Wei-Ge}.
Then, by a straightforward adaptation of the proof of \cite[Lemma 3.16]{GeLing2005b}, we obtain 5-GDDs of type $g^{15}$ for
$g \in \{6n, 6n + 4, \dots, 8n - 2\}$ whenever there exists a TD$(15,n)$ with odd $n$.
This interval contains 106 and 118 when $n = 17$, and 178 when $n = 23$.
\adfQED


\section{GDDs with block size 5 and type $g^u m^1$}
\label{sec:GDDs-with-block-size-5-and-type-g-u-m-1}

Assuming they might be of some use for future research,
we collect together an assortment of directly constructed 5-GDDs of type $g^u m^1$ that we have found during our investigations.
\begin{theorem}\label{thm:5-GDD-g-u-m-1-direct}
There exist $5$-$\mathrm{GDD}$s of types
$ 2^{36} 10^1 $,
$ 4^{12} 8^1 $,
$ 4^{21} 20^1 $,
$ 4^{22} 8^1 $,
$ 4^{24} 24^1 $,
$ 4^{25} 12^1 $,
$ 4^{26} 20^1 $,
$ 4^{27} 8^1 $,
$ 4^{30} 12^1 $,
$ 4^{31} 20^1 $,
$ 4^{32} 8^1 $,
$ 4^{37} 8^1 $,
$ 6^{12} 2^1 $,
$ 7^{20} 19^1 $,
$ 8^{10} 4^1 $,
$ 8^{12} 16^1 $,
$ 8^{13} 12^1 $,
$ 8^{18} 12^1 $,
$ 8^{20} 4^1 $,
$ 8^{20} 24^1 $,
$ 16^8 24^1 $ and
$ 24^7 8^1 $.
\end{theorem}
\begin{proof}
%

\adfhide{
$ 2^{36} 10^1 $,
$ 4^{12} 8^1 $,
$ 4^{21} 20^1 $,
$ 4^{22} 8^1 $,
$ 4^{24} 24^1 $,
$ 4^{25} 12^1 $,
$ 4^{26} 20^1 $,
$ 4^{27} 8^1 $,
$ 4^{30} 12^1 $,
$ 4^{31} 20^1 $,
$ 4^{32} 8^1 $,
$ 4^{37} 8^1 $,
$ 6^{12} 2^1 $,
$ 7^{20} 19^1 $,
$ 8^{10} 4^1 $,
$ 8^{12} 16^1 $,
$ 8^{13} 12^1 $,
$ 8^{18} 12^1 $,
$ 8^{20} 4^1 $,
$ 8^{20} 24^1 $,
$ 16^8 24^1 $ and
$ 24^7 8^1 $.
}

\adfDgap
\noindent{\boldmath $ 2^{36} 10^{1} $}~
With the point set $\{0, 1, \dots, 81\}$ partitioned into
 residue classes modulo $36$ for $\{0, 1, \dots, 71\}$, and
 $\{72, 73, \dots, 81\}$,
 the design is generated from

\adfLgap {\adfBfont
$\{21, 76, 30, 35, 0\}$,
$\{38, 9, 33, 7, 30\}$,
$\{65, 23, 8, 15, 4\}$,\adfsplit
$\{32, 79, 55, 30, 61\}$,
$\{72, 63, 9, 64, 54\}$,
$\{1, 35, 80, 60, 34\}$,\adfsplit
$\{9, 61, 28, 21, 65\}$,
$\{6, 12, 28, 40, 60\}$,
$\{0, 14, 59, 69, 73\}$

}
\adfLgap \noindent by the mapping:
$x \mapsto x + 2 j \adfmod{72}$ for $x < 72$,
$x \mapsto (x - 72 + 5 j \adfmod{10}) + 72$ for $x \ge 72$,
$0 \le j < 36$.
\ADFvfyParStart{(82, ((9, 36, ((72, 2), (10, 5)))), ((2, 36), (10, 1)))} 

\adfDgap
\noindent{\boldmath $ 4^{12} 8^{1} $}~
With the point set $\{0, 1, \dots, 55\}$ partitioned into
 residue classes modulo $12$ for $\{0, 1, \dots, 47\}$, and
 $\{48, 49, \dots, 55\}$,
 the design is generated from

\adfLgap {\adfBfont
$\{0, 1, 3, 11, 32\}$,
$\{0, 4, 9, 34, 48\}$,
$\{0, 6, 26, 41, 51\}$

}
\adfLgap \noindent by the mapping:
$x \mapsto x +  j \adfmod{48}$ for $x < 48$,
$x \mapsto (x +  j \adfmod{8}) + 48$ for $x \ge 48$,
$0 \le j < 48$.
\ADFvfyParStart{(56, ((3, 48, ((48, 1), (8, 1)))), ((4, 12), (8, 1)))} 

\adfDgap
\noindent{\boldmath $ 4^{21} 20^{1} $}~
With the point set $\{0, 1, \dots, 103\}$ partitioned into
 residue classes modulo $21$ for $\{0, 1, \dots, 83\}$, and
 $\{84, 85, \dots, 103\}$,
 the design is generated from

\adfLgap {\adfBfont
$\{84, 57, 27, 22, 28\}$,
$\{13, 86, 79, 82, 32\}$,
$\{6, 65, 52, 63, 92\}$,\adfsplit
$\{98, 1, 10, 24, 63\}$,
$\{77, 44, 3, 70, 85\}$,
$\{0, 4, 16, 56, 64\}$

}
\adfLgap \noindent by the mapping:
$x \mapsto x +  j \adfmod{84}$ for $x < 84$,
$x \mapsto (x - 84 + 5 j \adfmod{20}) + 84$ for $x \ge 84$,
$0 \le j < 84$.
\ADFvfyParStart{(104, ((6, 84, ((84, 1), (20, 5)))), ((4, 21), (20, 1)))} 

\adfDgap
\noindent{\boldmath $ 4^{22} 8^{1} $}~
With the point set $\{0, 1, \dots, 95\}$ partitioned into
 residue classes modulo $22$ for $\{0, 1, \dots, 87\}$, and
 $\{88, 89, \dots, 95\}$,
 the design is generated from

\adfLgap {\adfBfont
$\{70, 92, 1, 27, 15\}$,
$\{89, 73, 66, 69, 39\}$,
$\{53, 55, 5, 76, 37\}$,\adfsplit
$\{43, 53, 12, 18, 54\}$,
$\{0, 5, 13, 64, 73\}$

}
\adfLgap \noindent by the mapping:
$x \mapsto x +  j \adfmod{88}$ for $x < 88$,
$x \mapsto (x +  j \adfmod{8}) + 88$ for $x \ge 88$,
$0 \le j < 88$.
\ADFvfyParStart{(96, ((5, 88, ((88, 1), (8, 1)))), ((4, 22), (8, 1)))} 

\adfDgap
\noindent{\boldmath $ 4^{24} 24^{1} $}~
With the point set $\{0, 1, \dots, 119\}$ partitioned into
 residue classes modulo $24$ for $\{0, 1, \dots, 95\}$, and
 $\{96, 97, \dots, 119\}$,
 the design is generated from

\adfLgap {\adfBfont
$\{72, 99, 61, 64, 49\}$,
$\{13, 65, 45, 43, 101\}$,
$\{85, 72, 109, 44, 54\}$,\adfsplit
$\{38, 106, 84, 77, 63\}$,
$\{26, 53, 102, 87, 93\}$,
$\{112, 82, 46, 65, 91\}$,\adfsplit
$\{0, 1, 5, 38, 54\}$

}
\adfLgap \noindent by the mapping:
$x \mapsto x +  j \adfmod{96}$ for $x < 96$,
$x \mapsto (x +  j \adfmod{24}) + 96$ for $x \ge 96$,
$0 \le j < 96$.
\ADFvfyParStart{(120, ((7, 96, ((96, 1), (24, 1)))), ((4, 24), (24, 1)))} 

\adfDgap
\noindent{\boldmath $ 4^{25} 12^{1} $}~
With the point set $\{0, 1, \dots, 111\}$ partitioned into
 residue classes modulo $25$ for $\{0, 1, \dots, 99\}$, and
 $\{100, 101, \dots, 111\}$,
 the design is generated from

\adfLgap {\adfBfont
$\{30, 84, 46, 16, 72\}$,
$\{0, 110, 99, 94, 9\}$,
$\{58, 105, 77, 60, 95\}$,\adfsplit
$\{1, 44, 67, 109, 22\}$,
$\{8, 4, 75, 28, 35\}$,
$\{0, 3, 11, 39, 52\}$

}
\adfLgap \noindent by the mapping:
$x \mapsto x +  j \adfmod{100}$ for $x < 100$,
$x \mapsto (x - 100 + 3 j \adfmod{12}) + 100$ for $x \ge 100$,
$0 \le j < 100$.
\ADFvfyParStart{(112, ((6, 100, ((100, 1), (12, 3)))), ((4, 25), (12, 1)))} 

\adfDgap
\noindent{\boldmath $ 4^{26} 20^{1} $}~
With the point set $\{0, 1, \dots, 123\}$ partitioned into
 residue classes modulo $26$ for $\{0, 1, \dots, 103\}$, and
 $\{104, 105, \dots, 123\}$,
 the design is generated from

\adfLgap {\adfBfont
$\{23, 9, 54, 110, 8\}$,
$\{61, 123, 58, 88, 67\}$,
$\{104, 2, 9, 95, 72\}$,\adfsplit
$\{17, 78, 27, 122, 56\}$,
$\{76, 106, 59, 22, 57\}$,
$\{13, 68, 21, 26, 92\}$,\adfsplit
$\{0, 4, 16, 36, 64\}$

}
\adfLgap \noindent by the mapping:
$x \mapsto x +  j \adfmod{104}$ for $x < 104$,
$x \mapsto (x - 104 + 5 j \adfmod{20}) + 104$ for $x \ge 104$,
$0 \le j < 104$.
\ADFvfyParStart{(124, ((7, 104, ((104, 1), (20, 5)))), ((4, 26), (20, 1)))} 

\adfDgap
\noindent{\boldmath $ 4^{27} 8^{1} $}~
With the point set $\{0, 1, \dots, 115\}$ partitioned into
 residue classes modulo $27$ for $\{0, 1, \dots, 107\}$, and
 $\{108, 109, \dots, 115\}$,
 the design is generated from

\adfLgap {\adfBfont
$\{39, 18, 67, 92, 87\}$,
$\{114, 43, 30, 72, 65\}$,
$\{1, 63, 77, 73, 79\}$,\adfsplit
$\{44, 113, 1, 62, 59\}$,
$\{69, 29, 5, 38, 46\}$,
$\{0, 1, 12, 38, 57\}$

}
\adfLgap \noindent by the mapping:
$x \mapsto x +  j \adfmod{108}$ for $x < 108$,
$x \mapsto (x - 108 + 2 j \adfmod{8}) + 108$ for $x \ge 108$,
$0 \le j < 108$.
\ADFvfyParStart{(116, ((6, 108, ((108, 1), (8, 2)))), ((4, 27), (8, 1)))} 

\adfDgap
\noindent{\boldmath $ 4^{30} 12^{1} $}~
With the point set $\{0, 1, \dots, 131\}$ partitioned into
 residue classes modulo $30$ for $\{0, 1, \dots, 119\}$, and
 $\{120, 121, \dots, 131\}$,
 the design is generated from

\adfLgap {\adfBfont
$\{72, 2, 10, 117, 78\}$,
$\{108, 2, 104, 128, 35\}$,
$\{30, 116, 7, 27, 68\}$,\adfsplit
$\{8, 102, 75, 104, 92\}$,
$\{130, 83, 90, 115, 41\}$,
$\{27, 83, 6, 28, 43\}$,\adfsplit
$\{0, 9, 66, 101, 127\}$

}
\adfLgap \noindent by the mapping:
$x \mapsto x +  j \adfmod{120}$ for $x < 120$,
$x \mapsto (x +  j \adfmod{12}) + 120$ for $x \ge 120$,
$0 \le j < 120$.
\ADFvfyParStart{(132, ((7, 120, ((120, 1), (12, 1)))), ((4, 30), (12, 1)))} 

\adfDgap
\noindent{\boldmath $ 4^{31} 20^{1} $}~
With the point set $\{0, 1, \dots, 143\}$ partitioned into
 residue classes modulo $31$ for $\{0, 1, \dots, 123\}$, and
 $\{124, 125, \dots, 143\}$,
 the design is generated from

\adfLgap {\adfBfont
$\{106, 98, 1, 50, 64\}$,
$\{76, 59, 53, 127, 6\}$,
$\{92, 140, 19, 9, 54\}$,\adfsplit
$\{23, 101, 44, 133, 110\}$,
$\{65, 67, 80, 141, 62\}$,
$\{124, 112, 18, 17, 43\}$,\adfsplit
$\{100, 93, 8, 19, 41\}$,
$\{0, 4, 16, 44, 104\}$

}
\adfLgap \noindent by the mapping:
$x \mapsto x +  j \adfmod{124}$ for $x < 124$,
$x \mapsto (x - 124 + 5 j \adfmod{20}) + 124$ for $x \ge 124$,
$0 \le j < 124$.
\ADFvfyParStart{(144, ((8, 124, ((124, 1), (20, 5)))), ((4, 31), (20, 1)))} 

\adfDgap
\noindent{\boldmath $ 4^{32} 8^{1} $}~
With the point set $\{0, 1, \dots, 135\}$ partitioned into
 residue classes modulo $32$ for $\{0, 1, \dots, 127\}$, and
 $\{128, 129, \dots, 135\}$,
 the design is generated from

\adfLgap {\adfBfont
$\{95, 104, 64, 21, 74\}$,
$\{113, 53, 79, 56, 12\}$,
$\{20, 22, 4, 70, 42\}$,\adfsplit
$\{22, 113, 125, 126, 121\}$,
$\{42, 57, 130, 6, 115\}$,
$\{72, 133, 3, 55, 66\}$,\adfsplit
$\{0, 7, 42, 56, 89\}$

}
\adfLgap \noindent by the mapping:
$x \mapsto x +  j \adfmod{128}$ for $x < 128$,
$x \mapsto (x +  j \adfmod{8}) + 128$ for $x \ge 128$,
$0 \le j < 128$.
\ADFvfyParStart{(136, ((7, 128, ((128, 1), (8, 1)))), ((4, 32), (8, 1)))} 

\adfDgap
\noindent{\boldmath $ 4^{37} 8^{1} $}~
With the point set $\{0, 1, \dots, 155\}$ partitioned into
 residue classes modulo $37$ for $\{0, 1, \dots, 147\}$, and
 $\{148, 149, \dots, 155\}$,
 the design is generated from

\adfLgap {\adfBfont
$\{86, 63, 119, 43, 147\}$,
$\{21, 22, 53, 15, 13\}$,
$\{50, 34, 120, 37, 130\}$,\adfsplit
$\{149, 122, 88, 53, 147\}$,
$\{14, 133, 112, 150, 3\}$,
$\{60, 13, 9, 108, 135\}$,\adfsplit
$\{32, 113, 47, 77, 89\}$,
$\{0, 5, 19, 90, 107\}$

}
\adfLgap \noindent by the mapping:
$x \mapsto x +  j \adfmod{148}$ for $x < 148$,
$x \mapsto (x - 148 + 2 j \adfmod{8}) + 148$ for $x \ge 148$,
$0 \le j < 148$.
\ADFvfyParStart{(156, ((8, 148, ((148, 1), (8, 2)))), ((4, 37), (8, 1)))} 

\adfDgap
\noindent{\boldmath $ 6^{12} 2^{1} $}~
With the point set $\{0, 1, \dots, 73\}$ partitioned into
 residue classes modulo $12$ for $\{0, 1, \dots, 71\}$, and
 $\{72, 73\}$,
 the design is generated from

\adfLgap {\adfBfont
$\{32, 70, 25, 41, 21\}$,
$\{14, 31, 46, 56, 0\}$,
$\{9, 11, 48, 39, 70\}$,\adfsplit
$\{64, 58, 60, 41, 63\}$,
$\{26, 55, 21, 34, 54\}$,
$\{57, 72, 32, 47, 50\}$,\adfsplit
$\{0, 19, 37, 45, 51\}$

}
\adfLgap \noindent by the mapping:
$x \mapsto x + 2 j \adfmod{72}$ for $x < 72$,
$x \mapsto (x +  j \adfmod{2}) + 72$ for $x \ge 72$,
$0 \le j < 36$.
\ADFvfyParStart{(74, ((7, 36, ((72, 2), (2, 1)))), ((6, 12), (2, 1)))} 

\adfDgap
\noindent{\boldmath $ 7^{20} 19^{1} $}~
With the point set $\{0, 1, \dots, 158\}$ partitioned into
 residue classes modulo $19$ for $\{0, 1, \dots, 132\}$,
 $\{133, 134, \dots, 139\}$, and
 $\{140, 141, \dots, 158\}$,
 the design is generated from

\adfLgap {\adfBfont
$\{64, 48, 14, 54, 115\}$,
$\{39, 2, 156, 51, 94\}$,
$\{39, 101, 24, 128, 21\}$,\adfsplit
$\{0, 4, 91, 98, 145\}$,
$\{0, 1, 14, 22, 147\}$,
$\{0, 2, 25, 30, 88\}$,\adfsplit
$\{0, 17, 48, 81, 133\}$,
$\{0, 9, 68, 122, 140\}$,
$\{0, 24, 97, 138, 158\}$

}
\adfLgap \noindent by the mapping:
$x \mapsto x +  j \adfmod{133}$ for $x < 133$,
$x \mapsto (x +  j \adfmod{7}) + 133$ for $133 \le x < 140$,
$x \mapsto (x - 140 +  j \adfmod{19}) + 140$ for $x \ge 140$,
$0 \le j < 133$.
\ADFvfyParStart{(159, ((9, 133, ((133, 1), (7, 1), (19, 1)))), ((7, 19), (7, 1), (19, 1)))} 

\adfDgap
\noindent{\boldmath $ 8^{10} 4^{1} $}~
With the point set $\{0, 1, \dots, 83\}$ partitioned into
 residue classes modulo $10$ for $\{0, 1, \dots, 79\}$, and
 $\{80, 81, 82, 83\}$,
 the design is generated from

\adfLgap {\adfBfont
$\{56, 2, 24, 70, 3\}$,
$\{80, 42, 19, 60, 57\}$,
$\{14, 49, 6, 30, 77\}$,\adfsplit
$\{0, 2, 6, 31, 75\}$

}
\adfLgap \noindent by the mapping:
$x \mapsto x +  j \adfmod{80}$ for $x < 80$,
$x \mapsto (x +  j \adfmod{4}) + 80$ for $x \ge 80$,
$0 \le j < 80$.
\ADFvfyParStart{(84, ((4, 80, ((80, 1), (4, 1)))), ((8, 10), (4, 1)))} 

\adfDgap
\noindent{\boldmath $ 8^{12} 16^{1} $}~
With the point set $\{0, 1, \dots, 111\}$ partitioned into
 residue classes modulo $12$ for $\{0, 1, \dots, 95\}$, and
 $\{96, 97, \dots, 111\}$,
 the design is generated from

\adfLgap {\adfBfont
$\{34, 42, 100, 36, 59\}$,
$\{92, 89, 55, 85, 36\}$,
$\{88, 3, 12, 66, 103\}$,\adfsplit
$\{111, 28, 66, 56, 1\}$,
$\{43, 4, 22, 48, 108\}$,
$\{0, 1, 14, 46, 81\}$

}
\adfLgap \noindent by the mapping:
$x \mapsto x +  j \adfmod{96}$ for $x < 96$,
$x \mapsto (x +  j \adfmod{16}) + 96$ for $x \ge 96$,
$0 \le j < 96$.
\ADFvfyParStart{(112, ((6, 96, ((96, 1), (16, 1)))), ((8, 12), (16, 1)))} 

\adfDgap
\noindent{\boldmath $ 8^{13} 12^{1} $}~
With the point set $\{0, 1, \dots, 115\}$ partitioned into
 residue classes modulo $13$ for $\{0, 1, \dots, 103\}$, and
 $\{104, 105, \dots, 115\}$,
 the design is generated from

\adfLgap {\adfBfont
$\{52, 16, 14, 24, 64\}$,
$\{38, 99, 70, 95, 79\}$,
$\{90, 5, 0, 109, 87\}$,\adfsplit
$\{41, 103, 10, 113, 68\}$,
$\{35, 2, 17, 72, 105\}$,
$\{0, 1, 7, 60, 81\}$

}
\adfLgap \noindent by the mapping:
$x \mapsto x +  j \adfmod{104}$ for $x < 104$,
$x \mapsto (x - 104 + 3 j \adfmod{12}) + 104$ for $x \ge 104$,
$0 \le j < 104$.
\ADFvfyParStart{(116, ((6, 104, ((104, 1), (12, 3)))), ((8, 13), (12, 1)))} 

\adfDgap
\noindent{\boldmath $ 8^{18} 12^{1} $}~
With the point set $\{0, 1, \dots, 155\}$ partitioned into
 residue classes modulo $18$ for $\{0, 1, \dots, 143\}$, and
 $\{144, 145, \dots, 155\}$,
 the design is generated from

\adfLgap {\adfBfont
$\{49, 57, 14, 17, 15\}$,
$\{137, 122, 77, 61, 55\}$,
$\{52, 21, 14, 65, 150\}$,\adfsplit
$\{56, 79, 60, 23, 32\}$,
$\{6, 84, 32, 11, 59\}$,
$\{53, 12, 92, 152, 142\}$,\adfsplit
$\{2, 71, 13, 83, 100\}$,
$\{0, 10, 30, 95, 149\}$

}
\adfLgap \noindent by the mapping:
$x \mapsto x +  j \adfmod{144}$ for $x < 144$,
$x \mapsto (x +  j \adfmod{12}) + 144$ for $x \ge 144$,
$0 \le j < 144$.
\ADFvfyParStart{(156, ((8, 144, ((144, 1), (12, 1)))), ((8, 18), (12, 1)))} 

\adfDgap
\noindent{\boldmath $ 8^{20} 4^{1} $}~
With the point set $\{0, 1, \dots, 163\}$ partitioned into
 residue classes modulo $20$ for $\{0, 1, \dots, 159\}$, and
 $\{160, 161, 162, 163\}$,
 the design is generated from

\adfLgap {\adfBfont
$\{70, 95, 117, 58, 51\}$,
$\{9, 133, 124, 148, 61\}$,
$\{88, 99, 57, 3, 89\}$,\adfsplit
$\{67, 144, 10, 136, 14\}$,
$\{13, 117, 94, 123, 156\}$,
$\{15, 66, 80, 64, 148\}$,\adfsplit
$\{56, 99, 10, 38, 51\}$,
$\{0, 3, 58, 93, 160\}$

}
\adfLgap \noindent by the mapping:
$x \mapsto x +  j \adfmod{160}$ for $x < 160$,
$x \mapsto (x +  j \adfmod{4}) + 160$ for $x \ge 160$,
$0 \le j < 160$.
\ADFvfyParStart{(164, ((8, 160, ((160, 1), (4, 1)))), ((8, 20), (4, 1)))} 

\adfDgap
\noindent{\boldmath $ 8^{20} 24^{1} $}~
With the point set $\{0, 1, \dots, 183\}$ partitioned into
 residue classes modulo $20$ for $\{0, 1, \dots, 159\}$, and
 $\{160, 161, \dots, 183\}$,
 the design is generated from

\adfLgap {\adfBfont
$\{142, 54, 150, 133, 40\}$,
$\{172, 8, 137, 115, 2\}$,
$\{112, 17, 6, 69, 153\}$,\adfsplit
$\{72, 114, 39, 181, 129\}$,
$\{78, 137, 183, 114, 116\}$,
$\{46, 19, 145, 176, 108\}$,\adfsplit
$\{89, 40, 179, 43, 134\}$,
$\{125, 52, 120, 42, 174\}$,
$\{35, 54, 6, 36, 140\}$,\adfsplit
$\{0, 4, 16, 125, 132\}$

}
\adfLgap \noindent by the mapping:
$x \mapsto x +  j \adfmod{160}$ for $x < 160$,
$x \mapsto (x - 160 + 3 j \adfmod{24}) + 160$ for $x \ge 160$,
$0 \le j < 160$.
\ADFvfyParStart{(184, ((10, 160, ((160, 1), (24, 3)))), ((8, 20), (24, 1)))} 

\adfDgap
\noindent{\boldmath $ 16^{8} 24^{1} $}~
With the point set $\{0, 1, \dots, 151\}$ partitioned into
 residue classes modulo $8$ for $\{0, 1, \dots, 127\}$, and
 $\{128, 129, \dots, 151\}$,
 the design is generated from

\adfLgap {\adfBfont
$\{62, 141, 95, 9, 19\}$,
$\{94, 11, 93, 55, 146\}$,
$\{18, 115, 0, 15, 148\}$,\adfsplit
$\{30, 77, 9, 23, 96\}$,
$\{137, 31, 22, 81, 101\}$,
$\{3, 80, 106, 102, 135\}$,\adfsplit
$\{40, 70, 3, 5, 97\}$,
$\{0, 5, 11, 116, 145\}$

}
\adfLgap \noindent by the mapping:
$x \mapsto x +  j \adfmod{128}$ for $x < 128$,
$x \mapsto (x - 128 + 3 j \adfmod{24}) + 128$ for $x \ge 128$,
$0 \le j < 128$.
\ADFvfyParStart{(152, ((8, 128, ((128, 1), (24, 3)))), ((16, 8), (24, 1)))} 

\adfDgap
\noindent{\boldmath $ 24^{7} 8^{1} $}~
With the point set $\{0, 1, \dots, 175\}$ partitioned into
 residue classes modulo $7$ for $\{0, 1, \dots, 167\}$, and
 $\{168, 169, \dots, 175\}$,
 the design is generated from

\adfLgap {\adfBfont
$\{135, 1, 159, 70, 81\}$,
$\{13, 63, 15, 54, 32\}$,
$\{159, 28, 29, 3, 114\}$,\adfsplit
$\{107, 162, 91, 87, 55\}$,
$\{127, 17, 12, 173, 104\}$,
$\{115, 161, 55, 88, 155\}$,\adfsplit
$\{90, 16, 24, 120, 133\}$,
$\{0, 3, 18, 47, 170\}$

}
\adfLgap \noindent by the mapping:
$x \mapsto x +  j \adfmod{168}$ for $x < 168$,
$x \mapsto (x +  j \adfmod{8}) + 168$ for $x \ge 168$,
$0 \le j < 168$.
\ADFvfyParStart{(176, ((8, 168, ((168, 1), (8, 1)))), ((24, 7), (8, 1)))} 
%
\end{proof}


\section*{ORCID}

\noindent A. D. Forbes     \url{https://orcid.org/0000-0003-3805-7056}


\end{document}